\documentclass[12pt,a4paper]{amsart}
\usepackage{amsmath,amsfonts,amsthm,amssymb,nameref,mathtools,mathbbol,mathrsfs,amscd}
\usepackage[top=3cm, bottom=2.5cm, left=2.5cm, right=2.5cm]{geometry} %page margins
\usepackage{tikz} %makes graphs
\usetikzlibrary{calc,automata} %lets tikz calculate coordinates,enables nodes
\usepackage{color} %we like drawings
\usepackage{verbatim}
\usepackage{xfrac}
\usepackage{enumerate}
\usepackage{pinlabel}
\usepackage{graphicx}
\usepackage{caption}
\usepackage{subcaption}
\usepackage{float}
\usepackage{amsmath}
\numberwithin{figure}{subsection}

\usepackage[pdftex]{hyperref}
\usepackage[all]{hypcap}
\usepackage{microtype} % combats overfull boxes

\usepackage{overpic}

\theoremstyle{plain}
\newtheorem{Theorem}{Theorem}[section]
\newtheorem{Definition}[Theorem]{Definition}
\newtheorem{Lemma}[Theorem]{Lemma}
\newtheorem{Proposition}[Theorem]{Proposition}
\newtheorem{Corollary}[Theorem]{Corollary}

\newtheorem{Example}[Theorem]{Example}
\newtheorem{Examples}[Theorem]{Examples}
\newtheorem{Non-examples}[Theorem]{Non-examples}

\begin{document}

\title{RELATIVE HYPERBOLICITY OF GRAPHICAL SMALL CANCELLATION GROUPS}
%Relative Hyperbolicity of Graphical Small Cancellation Groups

\author{Suzhen Han}

%\author{Wen-yuan Yang}
\address{Beijing International Center for Mathematical Research (BICMR), Beijing University, No. 5 Yiheyuan Road, Haidian District, Beijing, China}
\email{suzhenhan@pku.edu.cn}
%\email{yabziz@gmail.com}

\thanks{S. H. is supported by  NSFC under Grant No. 11771022 and No. 11871078.}


\begin{abstract}
A piece of a labelled graph $\Gamma$ defined by D. Gruber is a labelled path that embeds into $\Gamma$ in two essentially different ways.
We prove that graphical $Gr'(\frac{1}{6})$ small cancellation groups whose associated pieces have uniformly bounded length are relative hyperbolic. In fact, we show that the Cayley graph of such group presentation is asymptotically tree-graded with respect to the collection of all embedded components of the defining graph $\Gamma$, if and only if the pieces of $\Gamma$ are uniformly bounded. This implies the relative hyperbolicity by a result of C. Dru\c{t}u, D. Osin and M. Sapir.
\end{abstract}

\maketitle

\section{Introduction}

Relative hyperbolicity is first introduced by Gromov \cite{G87} to study various algebraic and geometric examples, such as fundamental groups of finite volume hyperbolic manifolds, small cancellation quotients of free products, etc. There are several equivalent definitions for relative hyperbolic groups proposed by different authors, among which we will use the one due to  C. Dru\c{t}u, D. Osin and M. Sapir. This is based on the notion of tree-graded structure of a metric space introduced in \cite{DS05}.

\begin{Definition}[tree-graded spaces]
Let ${X}$ be a complete geodesic metric space and let $\mathcal{P}$ be a collection of closed geodesic subsets (called pieces). We say ${X}$ is \textit{tree-graded with respect to} $\mathcal{P}$ if the following two properties are satisfied:
\begin{itemize}
\item[$(T_1)$] Every two different pieces have at most one common point.
\item[$(T_2)$] Every simple geodesic triangle in $\mathbb{X}$ is contained in one piece.
\end{itemize}
\end{Definition}

A metric space $X$ is called \textit{asymptotically tree-graded} with respect to a collection of subsets $\mathcal{A}$ if every asymptotic cone of $X$ is tree-graded with respect to the collection of limit sets of sequences in $\mathcal{A}$. See   Section \ref{subsection:AsymptoticCones} for precise definitions.
In \cite{DS05}, a characterization of relative hyperbolicity for finitely generated groups is presented,  which will be taken as the definition of relatively hyperbolic groups in the present paper. Namely, a group $G$ generated by $S$ is hyperbolic relative to some subgroups $H_1,\cdots,H_m$ if and only if the Cayley graph $\mathscr{G}(G,S)$ is asymptotically tree-graded with respect to the collection of all left cosets of $H_1,\cdots,H_m$. Subgroups $H_1,\cdots,H_m$ are usually called the \textit{peripheral subgroups}.

Throughout the paper,  all groups under consideration are assumed to be finitely generated.

The above characterization is further generalized in Dru\c{t}u \cite{D09} by  allowing to consider peripheral subsets instead of peripheral subgroups, and thus gives  a truely metric version of relative hyperbolicity.
 %\cite[Theorem 1.11 and Appendix]{DS05}.

\begin{Theorem}\cite[Proposition 5.1 and its proof]{D09} \label{thm:TreeGradeImplyRelHyper}
Let $G$ be a group generated by a finite set $S$. Assume that $\mathscr{G}(G,S)$ is asymptotically tree-graded with respect to a collection $\mathcal{B}$ of subsets and $G$ permutes those subsets in $\mathcal{B}$, then the following holds:
\begin{itemize}
\item there exist $B_1,\cdots,B_k\in\mathcal{B}$, such that $\mathcal{B}=\{gB_i:g\in G,1\leq i\leq k\}$;
\item $G$ is hyperbolic relative to some subgroups $H_1,\cdots,H_m$, such that every $H_i$ is the stabilizer subgroup of some $B_{j_i}$ for $1\leq j_i\leq k$ and the Hausdorff distance between $H_i$ and $B_{j_i}$ is bounded.
\end{itemize}
\end{Theorem}
If the assumption ``$G$ permutes the subsets in $\mathcal{B}$'' were removed, the conclusion that $G$ is relatively hyperbolic also holds, which is a deeper result of Dru\c{t}u \cite[Theorem 1.6]{D09}.

The goal of the present paper is to characterize relatively hyperbolic groups in the class of graphical small cancellation groups introduced by Gromov \cite{G03}. This is a generalization of classical small cancellation groups, which serves a powerful tool to construct groups with prescribed embedded subgraphs in their Cayley graphs.

Given a directed graph $\Gamma$ labeled by a finite set $S$ of letters, let $G(\Gamma)$ be the group presented by $\langle S| ~\text{labels of simple circles in}~ \Gamma\rangle$. A \textit{piece} is a labelled path that embeds into $\Gamma$ in two different ways up to label-preserving automorphisms of $\Gamma$. The $Gr'(\frac{1}{6})$ condition imposed on the graph is a generalized graphical version of $C'(\frac{1}{6})$ condition, which requires the length of each piece is less than $\frac{1}{6}$ of the length of simple circles containing this piece.  See Section \ref{subsection:GraphicalSmallCancellation} for precise definitions.

A fundamental fact of  graphical small cancellation theory  is that under the graphical $Gr'(\frac{1}{6})$ condition, every component  of $\Gamma$ admits a convex embedding into the Cayley graph of $G(\Gamma)$ (see Lemma \ref{lemma:EmbedConvex}).  It provides a rich source of finitely generated groups with exotic or extreme properties: Gromov monster groups with an infinite expander family of graphs in their Cayley graphs \cite{G03,AD08}, counterexamples to the Baum-Connes conjecture with coefficients \cite{HLS02}, etc.

Recently, Gruber and Sisto \cite{GS18} showed that non-elementary graphical $Gr(7)$-labelled, in particular, $Gr'(\frac{1}{6})$-labelled groups are acylindrically hyperbolic (a weaker condition than relative hyperbolicity). It is indicated  in \cite{GS18} that the Cayley graph of such group presentation is weekly hyperbolic relative to the embedded components of the defining graph. This suggested an analogy with relatively hyperbolic space: the embedded components of the defining graph corresponding to the peripheral regions of relatively hyperbolic space. This is also well informed in the work of Arzhantseva-Cashen-Gruber-Hume  \cite{ACGH19}. Informally, typical results there say that geodesics which have bounded penetration into embedded components are strongly contracting, hence they behave like hyperbolic geodesics. However, the first example of strongly contracting element being unstable under changing generating set is constructed there as well. This is different from the fact in relatively hyperbolic groups that every hyperbolic element is strongly contracting for any generating set. Therefore, it is an interesting question to understand which $Gr'(\frac{1}{6})$-labelled group is relatively hyperbolic.

The main result of this paper is the following characterization of the asymptotically tree-graded structure of $Gr'(\frac{1}{6})$-labelled groups with respect to the embedded components.

\begin{Theorem}\label{thm:RelativeHyperbolicity}
Let $\Gamma$ be a $Gr'(\frac{1}{6})$-labelled graph labelled by $S$. Then the Cayley graph $\mathscr{G}(G(\Gamma),S)$ is asymptotically tree-graded with respect to the collection $\mathcal{A}$ of all embedded components of $\Gamma$ if and only if all pieces of $\Gamma$ have uniformly bounded length.
\end{Theorem}

For a $Gr'(\frac{1}{6})$-labelled graph $\Gamma$, $G(\Gamma)$   permutes   embedded components in $\mathcal{A}$. If $A_i\in \mathcal{A}$ is an embedded component in $\mathscr{G}(G(\Gamma),S)$ corresponding to a component $\Gamma_i$ of $\Gamma$, then the stabilizer group $\mathrm{Stab}(A_i)$ is isomorphic to the label-preserving automorphism group of component $\Gamma_i$ by Lemma \ref{lemma:Aut=Stab}. Therefore by Theorem \ref{thm:TreeGradeImplyRelHyper}, we obtain the following  corollary.

\begin{Corollary}\label{cor:RelativeHyperbolicity}
Let $\Gamma$ be a $Gr'(\frac{1}{6})$-labelled graph labelled by $S$. If the pieces of $\Gamma$ have uniformly bounded length, then $\Gamma$ has only finitely many components up to label-preserving automorphism and $G(\Gamma)$ is hyperbolic relative to some subgroups $H_1,\cdots,H_m$, where  $H_i$ are the label-preserving automorphism groups of these components.
\end{Corollary}

The (relative) hyperbolicity has been considered by other authors in the literature. For instacne,  A. Pankrat'ev \cite{P99} and  M. Steenbock \cite[Theorem 1]{S15} studied the hyperbolicity of (graphical) small cancellation quotients of free products of hyperbolic groups. In his thesis, D. Gruber \cite[Theorem 2.9]{dG15Thesis} generalized their results to establish the relative hyperbolicity of the free products of groups subject to finite labelled graphs satisfying the graphical small cancellation condition.

\vspace{.5em}
\noindent\textbf{Outline of the proof.}
Assume that the pieces of $\Gamma$ are uniformly bounded.
Let $\mathcal{A}$ be the collection of all embedded components of $\Gamma$. We want to show $\mathscr{G}(G(\Gamma),S)$ is asymptotically tree-graded with respect to $\mathcal{A}$. Combining   \cite[Lemma 4.5]{DS05}   and \cite[Corollary 4.19]{D09} together, it is sufficient to verify that $\mathcal{A}$ satisfies the following list of geometric properties:
\begin{Lemma}\cite{DS05,D09}\label{lemma:tree-graded}
Let $(Y,d_Y)$ be a geodesic metric space and let $\mathcal{B}$ be a collection of subsets of $Y$. The metric space $Y$ is
asymptotically tree-graded with respect to $\mathcal{B}$ if and only if the following properties are satisfied:
\begin{itemize}
\item[($\Lambda_1$)] finite radius tubular neighborhoods of distinct elements in $\mathcal{B}$ are either disjoint or intersecting with uniformly bounded diameters.
\item[($\Lambda_2$)] a geodesic with endpoints at distance at most one third of its length from a set $B$ in $\mathcal{B}$ intersects a uniformly bounded radius tubular neighborhood of $B$.
\item[($\Omega_3$)] for any asymptotic cone $\mathrm{Con}_{\omega}(X; (e_n), (l_n))$, any simple triangle with edges limit geodesics is contained in a subset from $\mathcal{B}_\omega$.
\end{itemize}
\end{Lemma}
Note that the set $\mathcal{B}_\omega$ is the $\omega$-limit of $\mathcal{B}$. Limit geodesic is the one that can be represented as an $\omega$-limit of a sequence geodesics.

First of all, we follow the approach of \cite{ACGH19} by using  Strebel's classification of combinatorial geodesic triangles to show that $\mathcal{A}$ is a strongly contracting system. Thus, the properties $(\Lambda_1)$ and $(\Lambda_2)$ of $\mathcal{A}$ follow from the strongly contracting property.

To verify the property $(\Omega_3)$, we take a detour by first showing  the property $(\Omega_2)$: for any asymptotic cone $\mathrm{Con}_{\omega}(X; (e_n), (l_n))$, any simple bigon with edges limit geodesics is contained in a subset from $\mathcal{B}_\omega$.  The reason why we do not verify property $(\Omega_3)$ directly is that there is little  known   about the classification  of combinatorial hexagons. Note that any simple bigon with edges limit geodesics can be represented as an $\omega$-limit of a sequence `fat' simple geodesic quadrangles by the strategy in the proof  of \cite[Proposition 4.14]{D09}. However, the classification of `special' combinatorial quadrangles provided by  Arzhantseva-Cashen-Gruber-Hume \cite{ACGH19} shows us that there is no `fat' special combinatorial quadrangle. This gives us great help to show the property $(\Omega_2)$. Further analysis are carried out to obtain property $(\Omega_3)$.

The ``only if'' part of the theorem is obvious by the ($\Lambda_1$) property and the fact that pieces of $\Gamma$ can always be realized as the paths contained in the intersections of two different embedded components.

\vspace{.5em}
\noindent\textbf{The structure of this paper.}
Section \ref{section:Preliminaries} sets up the notations used in the present paper and contains preliminaries on graphical small cancellation groups, asymptotic cones, tree-graded metric spaces. The proof of Theorem  \ref{thm:RelativeHyperbolicity} takes up the whole Section \ref{section:proof} and we will verify in order the desired properties listed in Lemma \ref{lemma:tree-graded}. Some examples are also presented at the end to illustrate the applications of the main theorem.

\vspace{.5em}
\noindent\textbf{Acknowledgement.}
The author wishes to thank Prof. Wen-yuan Yang for suggesting the problem, many useful comments and helpful conversations. The author also would like to thank Prof. G. N. Arzhantseva for many useful suggestions on the writing, and thank Dr. A. Genevois for bringing to me the references \cite{S15} and \cite{dG15Thesis} and useful suggestions.

\section{Preliminaries} \label{section:Preliminaries}

Let $A$ be a subset in a metric space $(X,d)$. We denote by $N_\delta(A)$ (resp. $\overline{N}_\delta(A)$) the set $\{x\in X | d(x,A)<\delta\}$
(resp.  $\{x\in X | d(x,A)\leq\delta\}$), which we call the (resp.  closed) $\delta$-tubular neighborhood of $A$.
And we denote by $\mathrm{diam}(A)$ the diameter of $A$. Given a point $x\in X$, and a subset $Y\subseteq X$, let $\pi_Y(x)$ be the set of point $y$ in $Y$ such that $d(x,y)=d(x,Y)$. The \textit{projection} of $A$ to $Y$ is the set $\pi_Y(A):=\cup_{a\in A}\Pi_Y(a)$.

The path $\gamma$  in $X$ under consideration is always assumed to be rectifiable with arc-length parametrization $[0, |\gamma|]\to \gamma$, where $|\gamma|$ denotes the length of $\gamma$. Denote by $\gamma_-,\gamma_+$ the initial and terminal points of $\gamma$ respectively, if $\gamma_i$ has subscript, we will denote by $\gamma_{i-},\gamma_{i+}$ for simplicity. And we denote by $\overline{\gamma}$ the inverse of $\gamma$, i.e. $\overline{\gamma}$ has parametrization $\overline{\gamma}(t)=\gamma(|\gamma|-t)$. For any two parameters $a<b\in [0, |\gamma|]$,  we denote by $[\gamma(a),\gamma(b)]_\gamma:=\gamma([a,b])$ the closed subpath of $\gamma$ between $\gamma(a)$ and $\gamma(b)$, The symbols $(\gamma(a),\gamma(b))_\gamma:=\gamma((a,b))$, $[\gamma(a),\gamma(b))_\gamma:=\gamma([a,b))$ and $(\gamma(a),\gamma(b])_\gamma:=\gamma((a,b])$
are defined analogously. For any $x,y\in X$, we denote by $[x,y]$ a choice of geodesic in $X$ from $x$ to $y$.

\subsection{Graphical small cancellation}  \label{subsection:GraphicalSmallCancellation}

As a generalization of classical small cancellation theory, the main benefit of graphical small cancellation theory is to embed a series of graphs into the Cayley graph of a group.
It was introduced by Gromov \cite{G03}, and then was modified by various authors \cite{O06,AD08,dG15}.

We will state the graphical small cancellation conditions following \cite{dG15}, before which we describe the group defined by a labelled graph, which first appeared in \cite{RS87}.

Let $\Gamma$ be a possibly infinite, and possibly non-connected, directed graph. A \textit{labelling} of $\Gamma$ by a set $S$ is a map assigning to each edge an element of $S$. The label of an edge path is the word in $M(S)$ (the free monoid on $S\sqcup S^{-1}$) reading along the path. If the corresponding edge is traversed in the opposite direction, then we read the formal inverse latter of the given label, otherwise just read the labelled letter.

The labelling is \textit{reduced} if at every vertex $v$, any two oriented edges both originating from $v$ or both terminating at $v$ have distinct labels.
Let $R$ be the set of words in $F(S)$ (the free group on $S$) read on simple cycles in $\Gamma$. Reducibility implies that the words in $R$ are cyclically reduced.  The \textit{group $G(\Gamma)$ defined by $\Gamma$} is given by the presentation $\langle S|R\rangle$.

\begin{Definition}(Piece \cite[Definition 1.5]{dG15Thesis}) \label{def:piece}
Let $\Gamma, \alpha$ be reduced $S$-labelled graphs. Two label-preserving maps of labelled graphs $\psi_1,\psi_2: \alpha\rightarrow\Gamma$ are called \textit{essentially distinct} if there is no label-preserving automorphism $\eta$ of $\Gamma$ with $\psi_2=\eta\circ\psi_1$.
A \textit{piece} of $\Gamma$ is a labelled path $\alpha$ (treated as a labelled graph) that admits two essentially distinct label-preserving maps $\phi_1,\phi_2:\alpha\rightarrow\Gamma$.
\end{Definition}

\begin{Definition}[$Gr'(\frac{1}{6})$ and $C'(\frac{1}{6})$ conditions] \label{def:GraphicalSmallCancellation}
%Let $\lambda>0$.
A reduced $S$-labelled graph $\Gamma$ is $Gr'(\frac{1}{6})$-labelled if any piece $\alpha$ contained in a simple cycle $\mathcal{O}$ of $\Gamma$ satisfies $|\alpha|<\frac{1}{6}|\mathcal{O}|$.

 A presentation $\langle S|R\rangle$ satisfies the classical $C'(\frac{1}{6})$ condition if the graph, constructed as disjoint union of cycle graphs labelled by the words in $R$, is $Gr'(\frac{1}{6})$-labelled.
\end{Definition}

%In this paper, we always choose $\lambda=\frac{1}{6}$.

Let $Y$ be a geodesic metric space, a subset $A\subseteq Y$ is \textit{convex} if every geodesic segment with endpoints in $A$ is contained in $A$.

\begin{Lemma}\cite[Lemma 2.15]{GS18} \label{lemma:EmbedConvex}
Let $\Gamma_0$ be a component of a $Gr'(\frac{1}{6})$-labelled graph $\Gamma$. For any choice of base vertices $x\in \mathscr{G}(G(\Gamma),S)$, $y\in \Gamma_0$, there is a unique label-preserving map $f:\Gamma_0\rightarrow \mathscr{G}(G(\Gamma),S)$ with $f(y)=x$. Moreover, such $f$ is an isometric embedding and its image in $X$ is convex.
\end{Lemma}

\begin{Definition}(Embedded component \cite[Definition 2.5]{ACGH19})
The image of the label-preserving map given by Lemma \ref{lemma:EmbedConvex} is called an \textit{embedded component} of $\Gamma$.
\end{Definition}

For each component $\Gamma_i$ of $\Gamma$, choose a basepoint $y_i\in \Gamma_i$.
Let $\Gamma_i'$ be the embedded component of $\Gamma$ which is the image of the label-preserving map $f_i:\Gamma_i\rightarrow \mathscr{G}(G(\Gamma),S)$ with $f_i(y_i)=1_{G(\Gamma)}$ (the identity element of $G(\Gamma)$).
Then an arbitrary embedded component of $\Gamma$ in $X$ is always a $G(\Gamma)$-translate of some $\Gamma_i'$.

%The proof of the following result is an consequence of the definition of pieces in $\Gamma$, and is left to the interested reader.
The following observation is an consequence of the definition of pieces in $\Gamma$.

\begin{Lemma} \label{lemma:PieceIntersection}
A labelled path $p$ is a piece of $\Gamma$ if and only if $p$ can embed into the intersection of two different embedded components.
\end{Lemma}

The following result is probably  well-known. However,  the author cannot find an exact reference  in the literature, so a proof is provided here for completeness.

\begin{Lemma} \label{lemma:Aut=Stab}
Let $\Gamma_0$ be a  component of a $Gr'(\frac{1}{6})$-labelled graph $\Gamma$. For any label-preserving embedding $f:\Gamma_0\hookrightarrow\mathscr{G}(G(\Gamma),S)$ provided by Lemma \ref{lemma:EmbedConvex}, the label-preserving automorphism group $\mathrm{Aut}(\Gamma_0)$ of the component $\Gamma_0$ is isomorphic to the stabilizer group $\mathrm{Stab}(f(\Gamma_0))$  in $G(\Gamma)$ of the embedded component $f(\Gamma_0)$.
\end{Lemma}

\begin{proof}
For any two different label-preserving maps $f_1,f_2:\Gamma_0\rightarrow\mathrm{Cay}(G(\Gamma),S)$ provided by Lemma \ref{lemma:EmbedConvex}, the stabilizer groups $\mathrm{Stab}(f_1(\Gamma_0)),\mathrm{Stab}(f_2(\Gamma_0))$ are isomorphic. Thus, without loss of generality, we may assume that $1_{G(\Gamma)}\in f(\Gamma_0)$. Let $x_0$ be the vertex of $\Gamma_0$ such that $f(x_0)=1_{G(\Gamma)}$.

It is obvious that $f:\Gamma_0\rightarrow f(\Gamma_0)$ is a label-preserving isomorphism. This implies  a natural monomorphism:
\begin{align*}
\Phi: \mathrm{Stab}(f(\Gamma_0)) &\hookrightarrow \mathrm{Aut}(\Gamma_0) \\
g &\mapsto \Phi(g)= f^{-1}gf.
\end{align*}

Observe first that for any $h\in \mathrm{Aut}(\Gamma_0)$, $h$ is determined by $h(x_0)$.
Since $\Gamma$ is reduced, every two edges originated from (resp. terminated at) $h(x_0)$ have different labels. Hence an edge originated from $h(x_0)$ and labelled by $s\in S$ must be the image under $h$ of the unique edge originated from $x_0$ and labelled also by $s$. The same holds for edges terminated at $h(x_0)$. Since $\Gamma$ is connected, the observation follows from induction.

We now prove that $\Phi$ is surjective. If $g_h=f(h(x_0))$ is the corresponding vertex in the embedded component   $f(\Gamma_0)$ in $\mathscr{G}(G(\Gamma),S)$, then $g_h=g_h(1_{G(\Gamma)})$ translates  the edges adjacent to $1_{G(\Gamma)}$   in $f(\Gamma_0)$ to the ones adjacent  to the vertex $g_h$ {which must be contained in $f(\Gamma_0)$ as well}.
By induction again, we know that $g_h$ stabilizes $f(\Gamma_0)$, and thus $g_h\in\mathrm{Stab}(f(\Gamma_0))$. It is easy to see that $\Phi(g_h)=h$, so $\Phi$ is an isomorphism.
\end{proof}

\subsection{Combinatorial polygons}

Diagram is one of the main tools in small cancellation theory. In this subsection, we introduce several relevant notions for future discussions.  %is the so-called `van Kampen diagram'.

A (singular disk) \textit{diagram} is a finite, contractible, $2$-dimensional CW-complex embedded in $\mathbb{R}^2$. It is $S$-\textit{labelled} if its $1$-skeleton, when treated as graph, is endowed with a labelling by $S$. It is a diagram \textit{over} the presentation $\langle S|R\rangle$ if it is $S$-labelled and the word read on the boundary cycle of each $2$-cell belongs to $R$.

A diagram is \textit{simple} if it is homeomorphic to a disc. We call $1$-cells \textit{edges} and $2$-cells \textit{faces}.

If $D$ is a diagram over $\langle S|R\rangle$, then for any choice of base vertices $x\in \mathscr{G}(G(\Gamma),S)$ and $y\in D$, there exists a unique label-preserving map $g$ from $1$-skeleton of $D$ to $\mathscr{G}(G(\Gamma),S)$ with $g(y)=x$. The map $g$ need not be an immersion in general. In the following, we sometimes confuse a diagram with its image in $\mathscr{G}(G(\Gamma),S)$ to simplify notations, however, it is easy to distinguish them from context.

An \textit{arc} in a diagram $D$ is an embedded edge path whose interior vertices have valence $2$ and whose initial and terminal vertices have valence different from $2$. An arc is \textit{exterior} if it is contained in the boundary of $D$ otherwise it is \textit{interior}. A face of a diagram $D$ is called \textit{interior} if its boundary arcs are all interior arcs, otherwise it is called a \textit{boundary face}. A \textit{$(3,7)$-diagram} is a diagram such that every interior vertices have valence at least $3$ and the boundary of every interior face is consisting of at least $7$ arcs.

If $\Pi$ is a face in a diagram, its \textit{interior degree} $i(\Pi)$ is the number of interior arcs in its boundary $\partial\Pi$, and its \textit{exterior degree} $e(\Pi)$ is the number of exterior arcs in $\partial\Pi$.

\begin{Definition}(combinatorial polygon \cite[Definition 2.11]{GS18})
A \textit{combinatorial $n$-gon} $(D,(\gamma_i))$ is a $(3,7)$-diagram $D$
with a decomposition of $\partial D$ into $n$ reduced subpaths (called \textit{sides}) $\partial D=\gamma_0\gamma_1\cdots\gamma_{n-1}$, such that every boundary face $\Pi$ with $e(\Pi)=1$ for which the exterior arc in $\partial\Pi$ is contained in one of the $\gamma_i$ satisfies $i(\Pi)\geq 4$.
\end{Definition}

A valence $2$ vertex that belongs to two sides is called a \textit{distinguished vertex}.
A face who contains a distinguished vertex is called a \textit{distinguished face}.

We adopt the convention in \cite{ACGH19} that the ordering of the sides of a combinatorial $n$-gon is considered up to cyclic permutation, with subscripts modulo $n$. By convention, $2$-gons, $3$-gons, and $4$-gons will be called bigons, triangles, and quadrangles respectively.

Now we return to the setting of $Gr'(\frac{1}{6})$-labelled graph $\Gamma$.
The following lemma is useful in diagrams over graphical small cancellation presentations.

\begin{Lemma} \cite[Lemma 2.13]{dG15} \label{Lemma:diagram}
Let $\Gamma$ be a $Gr'(\frac{1}{6})$-labelled graph. If $w\in F(S)$ satisfy $w=1_{G(\Gamma)}$ in $G(\Gamma)$, then, there exists a diagram $D$ over $\langle S|R\rangle$ such that $\partial D$ is labelled by $w$ and every interior arc of $D$ is a piece.
\end{Lemma}

A geodesic $n$-gon $P$ in $\mathscr{G}(G(\Gamma),S)$ is a cycle path that is a concatenation of $n$ geodesic segments $\gamma_0',\cdots,\gamma_{n-1}'$, which are called \textit{sides} of $P$.

The numbers appeared in the definition of combinatorial polygons are motivated by the following:

\begin{Lemma}\cite[Proposition 3.5]{ACGH19} \label{Lemma:CombPolygon}
Let $P=\gamma_0'\cdots\gamma_{n-1}'$ be a geodesic $n$-gon in $\mathscr{G}(G(\Gamma),S)$. Then there exists a diagram $D$ over $\langle S|R\rangle$ whose interior arcs are pieces such that:
\begin{itemize}
\item $\partial D=\gamma_0\cdots\gamma_{n-1}$ and the word reads on $\gamma_i$ is the same as the one reads on $\gamma_i'$ for $0\leq i\leq n-1$,
\item $D$ is a combinatorial $n$-gon after forgetting interior vertices of valence $2$.
\end{itemize}
\end{Lemma}

\subsubsection{Combinatorial bigons and triangles}

In the remainder of this subsection, our discussions about diagrams are always combinatorial, not necessarily over presentations.

\begin{Theorem}(Strebel's classification, \cite[Theorem 43]{S90}). \label{thm:ClassificationForTriangle}
Let $D$ be a simple diagram.
\begin{itemize}
\item If $D$ is a combinatorial bigon, then $D$ has the form $\mathrm{I}_1$ as depicted in Figure \ref{figure:classification-for-triangle}.
\item If $D$ is a combinatorial triangle\footnote{In our definition of combinatorial polygon, the form $\mathrm{III}_2$ of Strebel's classification can not appear.},
    then $D$ has one of the forms $\mathrm{I}_2$, $\mathrm{I}_3$, $\mathrm{II}$, $\mathrm{III}$, $\mathrm{IV}$, or $\mathrm{V}$ as depicted in Figure \ref{figure:classification-for-triangle}. %the picture below.
\end{itemize}
\end{Theorem}

\vspace{-1.7em}
\setlength{\unitlength}{3.5cm}
\begin{picture}(1,1)
\qbezier(0, 0.4)(0.5,0.9)(1, 0.4)
\qbezier(0, 0.4)(0.5,-0.1)(1, 0.4)

\put(0.35,0.63){\line(0,-1){0.46}}
\put(0.65,0.63){\line(0,-1){0.46}}
\put(0.45,0.4){$\cdots$}

\put(0,0.7){$\mathrm{I}_1$}

%\caption{$\mathrm{I}_2$}

\put(1.2,0){\line(1,0){1}}
\put(1.2,0){\line(2,3){0.5}}
\put(2.2,0){\line(-2,3){0.5}}

\put(1.5,0.45){\line(1,0){0.4}}
\put(1.4,0.3){\line(1,0){0.6}}
\put(1.685,0.34){$\vdots$}

\put(1.2,0.7){$\mathrm{I}_2$}

\put(2.4,0){\line(1,0){1}}
\put(2.4,0){\line(2,3){0.5}}
\put(3.4,0){\line(-2,3){0.5}}

\put(2.6,0){\line(-2,3){0.1}}
\put(2.8,0){\line(-2,3){0.2}}
\put(2.54,0.1){$\cdots$}

\put(3,0){\line(2,3){0.2}}
\put(3.2,0){\line(2,3){0.1}}
\put(3.13,0.1){$\cdots$}

\put(2.4,0.7){$\mathrm{I}_3$}
\end{picture}

\vspace{-1.3em}
\setlength{\unitlength}{3.3cm}
\begin{picture}(1,1) \label{figure:classification-for-triangle}
\put(0,0){\line(1,0){1}}
\put(0,0){\line(2,3){0.5}}
\put(1,0){\line(-2,3){0.5}}

\put(0.2,0){\line(-2,3){0.1}}
\put(0.4,0){\line(-2,3){0.2}}
\put(0.14,0.1){$\cdots$}

\put(0.6,0){\line(2,3){0.2}}
\put(0.8,0){\line(2,3){0.1}}
\put(0.73,0.1){$\cdots$}

\put(0.3,0.45){\line(1,0){0.4}}
\put(0.4,0.6){\line(1,0){0.2}}
\put(0.485,0.49){$\vdots$}

\put(0,0.7){$\mathrm{II}$}

\put(1.2,0){\line(1,0){1}}
\put(1.2,0){\line(2,3){0.5}}
\put(2.2,0){\line(-2,3){0.5}}

\put(1.5,0){\line(-2,3){0.15}}
\put(1.7,0){\line(-2,3){0.25}}
\put(1.43,0.1){$\cdots$}

\put(1.7,0){\line(2,3){0.25}}
\put(1.9,0){\line(2,3){0.15}}
\put(1.84,0.1){$\cdots$}

\put(1.5,0.45){\line(1,0){0.4}}
\put(1.6,0.6){\line(1,0){0.2}}
\put(1.685,0.49){$\vdots$}

\put(1.2,0.7){$\mathrm{III}$}

\put(2.4,0){\line(1,0){1}}
\put(2.4,0){\line(2,3){0.5}}
\put(3.4,0){\line(-2,3){0.5}}

\put(2.6,0){\line(-2,3){0.1}}
\put(2.8,0){\line(-2,3){0.2}}
\put(2.54,0.1){$\cdots$}

\put(3,0){\line(2,3){0.2}}
\put(3.2,0){\line(2,3){0.1}}
\put(3.13,0.1){$\cdots$}

\put(2.7,0.45){\line(1,0){0.4}}
\put(2.8,0.6){\line(1,0){0.2}}
\put(2.885,0.49){$\vdots$}

\put(2.9,0.225){\line(0,1){0.225}}
\put(2.9,0.225){\line(2,-1){0.185}}
\put(2.9,0.225){\line(-2,-1){0.185}}

\put(2.4,0.7){$\mathrm{IV}$}

\put(3.6,0){\line(1,0){1}}
\put(3.6,0){\line(2,3){0.5}}
\put(4.6,0){\line(-2,3){0.5}}

\put(3.8,0){\line(-2,3){0.1}}
\put(4,0){\line(-2,3){0.2}}
\put(3.74,0.1){$\cdots$}

\put(4.2,0){\line(2,3){0.2}}
\put(4.4,0){\line(2,3){0.1}}
\put(4.33,0.1){$\cdots$}

\put(3.9,0.45){\line(1,0){0.4}}
\put(4,0.6){\line(1,0){0.2}}
\put(4.085,0.49){$\vdots$}

\put(4.1,0.225){\line(0,-1){0.225}}
\put(4.1,0.225){\line(2,1){0.26}}
\put(4.1,0.225){\line(-2,1){0.26}}

\put(3.6,0.7){$\mathrm{V}$}

\put(0.3,-0.2){F\textsc{igure} 2.2.1.~ Strebel's classification of combinatorial bigons and triangles.}
%\label{figure:classification-for-triangle}
\end{picture}
\vspace{2em}

\subsubsection{Special combinatorial quadrangles}

In this subsection we introduce some notions from \cite{ACGH19}.

A combinatorial $n$-gon $(D,(\gamma_i))$ is \textit{degenerate} if there exists $i$, such that $$(D, \gamma_0, \cdots,\gamma_i\gamma_{i+1}, \cdots, \gamma_{n-1})$$ is a combinatorial $(n-1)$-gon. In this case the terminal vertex of $\gamma_i$ is called a \textit{degenerate vertex}.

\begin{Definition}\cite[Definition 3.12]{ACGH19} %\cite[Definition 3.12]{ACGH19}
A combinatorial $n$-gon $D$ is \textit{reducible} if it admits a vertex or face reduction as defined below, otherwise it is \textit{irreducible}.
\end{Definition}

Let $(D,(\gamma_i))$ be a combinatorial $n$-gon. We denote by $\Pi^\circ$ (resp. $e^\circ$) the interior of a face $\Pi$ (resp. an edge $e$).

\noindent\textbf{Vertex reduction.}
If  $v\in D$ is a cut vertex such that $v$ is contained in exactly two boundary faces, then the closures of the two  components of  $D\setminus v$  denoted by $D', D''$ respectively are then subcomplexes of $D$  in a natural way. The \textit{vertex reduction of $(D,(\gamma_i))$ at $v$}  gives two combinatorial polygons $D',D''$ whose sides are either these $\gamma_i$  or subsegments of $\gamma_i$ cut out by $v$.

\noindent\textbf{Face reduction.}
Suppose $\Pi\subseteq D$ is a separating boundary face with  two boundary arcs $e,e'\subseteq \partial\Pi$   such that $D\setminus(\Pi^\circ\cup e^\circ\cup (e')^\circ)$ has exactly two components, and each contains a distinguished vertex. The \textit{face reduction of  $(D,(\gamma_i))$ at $(\Pi,e,e')$} produces two combinatorial polygons obtained by collapsing a simple path connecting $e$ and $e'$ whose interior contained in $\Pi^\circ$ to a vertex, and then performing vertex reduction at the resulting vertex.

The inverse operation of vertex (resp. face) reduction is called \textit{vertex (resp. face) combination}. Precise definitions see \cite{ACGH19}.

\vspace{.5em}
It is easy to see the definition of reducibility in \cite{ACGH19} is equivalent to that here.

A simple combinatorial $n$-gon, for $n>2$, is \textit{special} if it is non-degenerate, irreducible, and every non-distinguished vertex has valence $3$.

The classification of special combinatorial quadrangles provided by G. Arzhantseva, C. Cashen, D. Gruber, and D. Hume gives us the following useful observation.

\begin{Lemma} \label{lemma:SpecialCombGeodQuadrangles}
Let $D$ be a simple, non-degenerate, irreducible combinatorial quadrangle, then there exist two opposite sides which can be connected by a path consisting of at most $6$ interior arcs.
\end{Lemma}

\begin{proof}
The blowing up operations in the proof of \cite[Proposition 3.20]{ACGH19} produce a special combinatorial quadrangle $D'$ associated with $D$ which has the following properties:
\begin{itemize}
\item diagram $D'$ is the image of $D$ under the compositions of finitely many blow-up vertex maps,
\item the boundary of $D'$ is the same as $D$, i.e., those blowing up of the compositions restricted to $\partial D$ are identities,
\item compared with $D$, $D'$ only has more interior edges. %obtained by a series of
\end{itemize}

By the classification of special combinatorial quadrangles
\cite[Theorem 3.18]{ACGH19}, we know that there exist two opposite sides of $D'$ and a path $p'$ connecting them which is a concatenation of at most $6$ interior arcs.
Then there exists a path $p\subset D$ such that the image of $p$ under the compositions of blow-up maps above is $p'$,
in other words, $p'$ is obtained via a series blowing up vertex operations from $p$. Therefore, $p$ connects two opposite sides of $D$ corresponding to the two opposite sides of $D'$ connected by $p'$, and $p$ consists of interiors arcs no more than the number of those of $p'$. Hence $p$ is a concatenation of at most $6$ interior arcs of $D$.
\end{proof}

\subsection{Asymptotic cones of a metric space} \label{subsection:AsymptoticCones}

Asymptotic cone was first essentially used by Gromov in \cite{G81} and then formally introduced by van den Dries and Wilkie in \cite{dDW84}.

Recall that a \textit{non-principal ultrafilter} is a finitely additive measure $\omega$ on the set of all subsets of $\mathbb{N}$ (or, more generally, of a countable set) such that every subset has measure either $0$ or $1$ and all finite subsets have measure $0$. Throughout the paper all ultrafilters are non-principal.

Let $A_n$ and $B_n$ be two sequences of objects and let $\mathcal{R}$ be a relation between $A_n$ and $B_n$ for every $n$. We write $A_n\mathcal{R}_\omega B_n$ if $A_n\mathcal{R} B_n$ $\omega$-almost surely, that is,
\[\omega(\{n\in\mathbb{N}|A_n\mathcal{R} B_n\})=1\]
For example, $\in_\omega,=_\omega,<_\omega,\subseteq_\omega$.

Given a space $X$, the \textit{ultrapower} $X^\omega$ of $X$ is the quotient $X^{\mathbb{N}}/\approx$, where $(x_n)\approx(y_n)$ if $x_n=_\omega y_n$.

Let $(X,d)$ be a metric space, $\omega$ an ultrafilter over $\mathbb{N}$, $(e_n)$ an element in $X^{\omega}$, and $(l_n)$ a sequence of numbers with $\lim_\omega l_n=+\infty$.
Define
\[X_e^\omega=\{(x_n)\in X^{\mathbb{N}}| \lim_\omega\frac{d(x_n,e_n)}{l_n}<+\infty \}.\]
Define an equivalence relation $\sim$ on $X_e^\omega$:
\[(x_n)\sim(y_n)\Leftrightarrow \lim_\omega \frac{d(x_n,y_n)}{l_n}=0.\]

\begin{Definition}
The quotient $X_e^\omega/\sim$, denoted by $\mathrm{Con}_{\omega}(X; (e_n),(l_n))$, is called the \textit{asymptotic cone} of $X$ with respect to the ultrafilter $\omega$, the scaling constants $(l_n)$ and the observation point $(e_n)$. It admits a natural metric $d_\omega$ defined as following:
\[d_\omega((x_n),(y_n))=\lim_\omega\frac{d(x_n,y_n)}{l_n}.\]
\end{Definition}

For a sequence $(A_n)$ of subsets in $X$, its $\omega$-\textit{limit} in $\mathrm{Con}_{\omega}(X; (e_n),(l_n))$ is defined by
\[\lim_\omega A_n=\{\lim_\omega a_n|a_n\in_\omega A_n\}.\]
Obviously, if $\lim_\omega\frac{d(e_n,A_n)}{l_n}=+\infty$, then $\lim_\omega A_n$ is empty. It is easy to verify that every limit set $\lim_\omega A_n$ is closed.

Let $\alpha_n$ be a sequence of geodesics with length of order $O(l_n)$. Then the $\omega$-limit $\lim_\omega\alpha_n$ in $\mathrm{Con}_{\omega}(X; e,d)$ is either empty or a geodesic. If $\lim_\omega \alpha_n$ is a geodesic in $\mathrm{Con}_{\omega}(X; e,d)$, then we call it a \textit{limit geodesic}. Therefore, any asymptotic cone of a geodesic metric space is also a geodesic metric space.

Not every geodesic in $\mathrm{Con}_{\omega}(X; e,d)$ is a limit geodesic, not even in particular that $X$ is a Cayley graph of a finitely generated group (see \cite{D09} for a counterexample).

\subsection{Tree-graded metric spaces}

Let $(X,d)$ be a geodesic metric space and let $\mathcal{A}=\{A_\lambda|\lambda\in \Lambda\}$ be a collection of subsets in $X$. In every asymptotic cone $\mathrm{Con}_\omega(X;(e_n),(l_n))$, we consider the collection $\mathcal{A}_\omega$ of limit subsets
\[\{\lim_\omega A_{\lambda_n}|(\lambda_n)\in \Lambda^\omega,s.t., \lim_\omega\frac{d(e_n,A_{\lambda_n})}{l_n}<+\infty\}.\]

\begin{Definition} \cite{DS05}
The metric space $X$ is asymptotically tree-graded with respect to $\mathcal{A}$ if every asymptotic cone $\mathrm{Con}_\omega(X;(e_n),(l_n))$ of $X$ is tree-graded with respect to $\mathcal{A}_\omega$.
\end{Definition}

The notion of an asymptotically tree-graded metric space is equivalent to a list of geometric conditions (see \cite[Theorem 4.1 and Remark 4.2 (3)]{DS05}). Now we introduce some related geometric notions and results.

\begin{Definition}[fat polygons] \label{def:FatPolygons}
Let $n\ge 1$ be an integer.  A geodesic $n$-gon $P$ is $\theta$-fat if the distance between any two non-adjacent edges is greater than $\theta$.
\end{Definition}

Our notion of ``fat polygons" is weaker than the notion in \cite{D09} (cf. \cite[Lemma 4.6]{D09}). However the weak version is enough for our purpose.
The method of \cite[Proposition 4.14]{D09} can also apply to bigons, therefore:
\begin{Lemma}\label{lemma:FatQuadrangles}
Let $\theta>0$. In any asymptotic cone $\mathrm{Con}_{\omega}(X; (e_n),(l_n))$ of a geodesic metric space $(X,d)$,
every simple non-trivial bigon with edges limit geodesics is the limit set $\lim_\omega Q_n$ of a sequence $(Q_n)$ of simple geodesic quadrangles that are $\theta$-fat $\omega$-almost surely.
\end{Lemma}

Let $\alpha_1$ and $\alpha_2$ be two paths. A bigon \textit{formed by} $\alpha_1$ and $\alpha_2$ is a union of a subpath of $\alpha_1$ with a subpath of $\alpha_2$ such that the two subpaths have common endpoints.

\begin{Lemma}\cite[Lemma 3.8]{D09} \label{lemma:OneImplyTheOther}
Let $Y$ be a metric space and let $\mathcal{B}$ be a collection of closed subsets of $Y$ which satisfies property $(T_1)$.
Let $\alpha$ and $\beta$ be two paths with common endpoints such that any non-trivial simple bigon formed by $\alpha$ and $\beta$ is contained in a subset in $\mathcal{B}$.

If $\alpha$ is contained in $B\in \mathcal{B}$, then $\beta$ is also contained in $B$.
\end{Lemma}

\begin{Lemma}\cite[Proposition 3.9]{D09} \label{lemma:FourGeodesics}
Let $Y$ be a metric space and let $\mathcal{B}$ be a collection of closed subsets of $Y$ which satisfies property $(T_1)$.

Let $\alpha_1$ and $\beta_1$ (resp. $\alpha_2$ and $\beta_2$) be two paths with common endpoints $u,v$ (resp. $v,w$). Assume that:
\begin{enumerate}
\item $\alpha_1\cap\alpha_2=\{v\}$, $\beta_1\cap\beta_2$ contains a point $a\neq v$;
\item all non-trivial simple bigons formed either by $\beta_1$ and $\beta_2$, or by $\alpha_1$ and $\beta_1$, or by $\alpha_2$ and $\beta_2$ are contained in a subset in $\mathcal{B}$.
\end{enumerate}
Then the bigon formed by $\beta_1$ and $\beta_2$ with endpoints $a$ and $v$ is contained in a subset in $\mathcal{B}$.
\end{Lemma}

\section{The proof of the Theorem} \label{section:proof}

In this Section, $\Gamma$ is always a $Gr'(\frac{1}{6})$-labeled graph labeled by a finite set $S$, whose pieces have length less than $M$. Let $X:=\mathscr{G}(G(\Gamma),S)$ be the Cayley graph of $(G(\Gamma),S)$, with the induced length metric $d$ on $X$ by assigning each edge length $1$.
In this setting, we want to show $\mathscr{G}(G(\Gamma),S)$ is asymptotically tree-graded with respect to the collection of all embedded components of $\Gamma$ ( denoted by $\mathcal{A}$).

\subsection{Contracting property of embedded components}

\begin{Definition}[Contracting subset]
Let $(Y,d_Y)$ be a metric space.
A subset $A\subseteq Y$ is called strongly $(\kappa,C)$-\textit{contracting} if for any geodesic $\gamma$ disjoint with $\overline{N}_C(A)$, we have $\mathrm{diam}(\pi_A(\gamma))\leq \kappa$. A collection of strongly $(\kappa,C)$-contracting subsets is referred to a strongly $(\kappa,C)$-\textit{contracting system}.
\end{Definition}

\begin{Example} \label{Example:LeftCosetsInRelHyp}
In the Cayley graph of a relatively hyperbolic group, the collection of all left cosets of peripheral subgroups is a strongly contracting system by \cite[Proposition 8.2.4]{GP15}.
\end{Example}

Now we show that $\mathcal{A}$ is a strongly contracting system.

\begin{Proposition} \label{proposition:contraction}
Let $\Gamma$ be a $Gr'(\frac{1}{6})$-labelled graph whose pieces are uniformly bounded by $M$. Then the collection of all embedded components of $\Gamma$ in $\mathscr{G}(G(\Gamma),S)$ is a strongly $(2M,0)$-contracting system.
\end{Proposition}

The method in the proof of \cite[Theorem 4.1]{ACGH19} for geodesics can also apply to embedded components. However, in the case of an embedded component, the proof is much easier.

\begin{proof}
Let $\Gamma_0$ be an embedded component, and let $\alpha$ be a geodesic disjoint from $\Gamma_0$. If $x',y'\in \pi_{\Gamma_0}(\alpha)$ are points realizing $\mathrm{diam}(\pi_{\Gamma_0}(\alpha))$, we wish  to bound $d(x',y')$ by $2M$.

If $x'=y'$ we are done. Otherwise, choose $x,y\in\alpha$, such that $x'\in\pi_{\Gamma_0}(x),y'\in\pi_{\Gamma_0}(y)$. Let $\alpha'$ be the subpath of $\alpha$ or $\overline{\alpha}$ (the inverse of $\alpha$) from $x$ to $y$. Choose a path $p$ from $x'$ to $y'$ in $\Gamma_0$, and choose geodesics $\beta_1=[x,x']$ and $\beta_2=[y,y']$, then $\beta_1\cap\Gamma_0=x'$ and $\beta_2\cap\Gamma_0=y'$ since $x',y'$ are closest point projection. Choose a diagram $D$ over $\langle S|R\rangle$ as in Lemma \ref{Lemma:diagram} filling the quadrangle $\alpha'\beta_2\overline{p}\overline{\beta_1}$. Assume that we have chosen $p$ and $D$ so that $D$ has minimal number of edges among all possible choices of $p$. To simplify the notation, the sides of $D$ are also denoted by $\alpha',\beta_2,\overline{p},\overline{\beta_1}$.

First assume that $\beta_1$ and $\beta_2$ are not intersect. Then $\alpha'$ can be chosen in such a way that both $\alpha'\cap\beta_1$ and $\alpha'\cap\beta_2$ have only one point. Therefore   $\alpha'\beta_2\overline{p}\overline{\beta_1}=\partial D$ is a simple cycle, and $D$ is a simple combinatorial quadrangle.

Let $a$ be an arc of $D$ contained in $p$ and let $\Pi$ be a face of $D$ containing $p$. Then $a$ has a lift to $\Gamma$ via first being a subpath of $p\subseteq \Gamma_0$ then returning to a component of $\Gamma$. Another lift of $a$ to $\Gamma$ via being a subpath of $\partial\Pi$ whose label is contained in $R$, and thus $\partial\Pi$ has a embedding to $\Gamma$. If these two lifts coincide (up to label-preserving automorphisms of $\Gamma$), then we can remove the edges of $a$ from $D$, thus we obtain a path $p'$ in $\Gamma_0$ by replacing the subpath $a$ of $p$ with the interior edges of $\Pi$, which contradicts the minimality. Hence, the two lifts are essentially distinct, and $a$ is a piece. If $p_-$ (resp. $p_+$) is a distinguished vertex, $\Pi$ is the face of $D$ containing the initial (resp. terminal) edge of $p$ and $a$ is the initial (resp. terminal) subpath of $p$ contained in $\partial\Pi\cap p$, then the above two lifts are always essentially distinct since $\beta_1\cap\Gamma_0=x'$ (resp. $\beta_2\cap\Gamma_0=y'$). Therefore, $p$ is a concatenation of pieces.

We construct a new diagram $D'$ by attaching a new face $\Pi'$ to $D$ with a proper subpath $p'$ of $\partial\Pi'$ identifying the side $p$ of $D$.
This operation is purely combinatorial: the boundary of $\Pi'$ is not labelled. We claim that $D'$ is a combinatorial triangle with a distinguished vertex in $\partial\Pi'\setminus p'$. For any interior face of $D'$, it must be contained in $D$ and its boundary is always a concatenation of pieces in $\Gamma$ by the previous paragraph, thus its interior degree is at least $7$ by the   $Gr'(\frac{1}{6})$     condition. For any boundary face $\Pi$ of $D'$ with $e(\Pi)=1$ for which the exterior arc $b$ in $\partial\Pi$ contained in one side of $D'$, $\Pi$ also must be contained in $D$ and $b$ is contained in one of $\alpha',\overline{\beta_1},\beta_2$. So $b$ is a geodesic in $X$ and the interior arcs of $\Pi$ are pieces, thus the sum of the length of those interior arcs is at least half of the length of $\partial\Pi$, therefore $i(\Pi)\geq4$ by  the $Gr'(\frac{1}{6})$ condition again.

By discussing all the cases listed in Theorem \ref{thm:ClassificationForTriangle} about $D'$, we have that $p$ consists of at most $2$ pieces, so $d(x',y')$ is bounded by $2M$.

Now suppose that $\beta_1$ and $\beta_2$ intersect. Let $D''$ be the maximal simple subdiagram of $D$ containing $p$ (such $D''$ exists since $p$ is disjoint with $\alpha$ and $p$ intersects $\beta_1,\beta_2$ only in their endpoints).
Then $D''$ is a simple combinatorial triangle. The same argument as above shows that $p$ is a piece (here the corresponding $D'$ can only have form $\mathrm{I}_1$ in Theorem \ref{thm:ClassificationForTriangle}), so $d(x',y')$ is bounded by $2M$ as well.
\end{proof}

We call the simple cycle of $\mathscr{G}(G(\Gamma),S)$  contained in an embedded components an \textit{embedded simple cycle}. We know from the proof of \cite[Theorem 4.1]{ACGH19} that a geodesic is strongly contracting if and only if its intersections with all embedded simple cycles are uniformly bounded.
Similarly, an embedded component is strongly contracting provided that its intersections with all other embedded components are uniformly bounded. Furthermore, an embedded component is strongly contracting provided that all pieces are uniformly bounded because the paths contained in the intersection of two different embedded components are pieces (Lemma \ref{lemma:PieceIntersection}).
Since our proof strongly depends on the contractibility of the embedded components, the condition that all pieces of the defining graph have uniformly bounded length seems can not be improved if we use the method in this paper.

\begin{Lemma} \label{lemma:ASatisfyAlpha1}
$\mathcal{A}$ satisfies property $(\Lambda_1)$.
\end{Lemma}

\begin{proof}
Given two different embedded components $\Gamma_1, \Gamma_2\in \mathcal{A}$, and $\delta\geq0$. For any two point $x_0,y_0\in N_\delta(\Gamma_1)\cap N_\delta(\Gamma_2)$, there exists $x_1,y_1\in \Gamma_1$, such that $d(x_1,x_0)\leq\delta,d(y_1,y_0)\leq\delta$. By Lemma \ref{lemma:EmbedConvex}, the embedded components of $\Gamma$ is convex in $\mathscr{G}(G(\Gamma),S)$, in particular, $\Gamma_1$ is convex. Thus we have that $\alpha=[x_1,y_1]\subseteq \Gamma_1$.

If $\alpha$ is disjoint from $\Gamma_2$, then $\mathrm{diam}(\pi_{\Gamma_2}(\alpha))\leq 2M$ by Proposition \ref{proposition:contraction}. Let $a\in \pi_{\Gamma_2}(x_1),b\in\pi_{\Gamma_2}(y_1)$, then $d(a,b)\leq \mathrm{diam}(\pi_{\Gamma_2}(\alpha))\leq 2M$. And we have
$$d(x_1,a)\leq d(x_1,x_0)+d(x_0,\Gamma_2)\leq 2\delta,~d(y_1,b)\leq d(y_1,y_0)+d(y_0,\Gamma_2)\leq2\delta.$$
Therefore, $d(x_1,y_1)\leq d(x_1,a)+d(a,b)+d(b,y_1)\leq 2M+4\delta$, so $d(x_0,y_0)\leq d(x_0,x_1)+d(x_2,y_1)+d(y_1,x_0)\leq 2M+6\delta$.

Now assume that the intersection $\alpha\cap \Gamma_2$ is nonempty.
Let $z$ be the starting point of $\alpha\cap \Gamma_2$ if $x_1\notin \Gamma_2$, i.e., $z$ is the point in $\alpha\cap \Gamma_2$ such that $[x_1,z)_{\alpha}\cap \Gamma_2=\emptyset$.
Let $w$ be the ending point of $\alpha\cap \Gamma_2$ if $y_1\notin \Gamma_2$, i.e., $w$ is the point in $\alpha\cap \Gamma_2$ such that $(w,y_1]_{\alpha}\cap \Gamma_2=\emptyset$.
With $\alpha$ replaced by either $[x_1,z]_{\alpha}$ or $[w,y_1]_{\alpha}$, the discussion in the previous paragraph still holds, so we have $d(x_1,z)\leq 2M+4\delta, d(w, y_1)\leq 2M+4\delta$.
Since $z,w\in \Gamma_1\cap \Gamma_2$, and $\Gamma_1,\Gamma_2$ are convex, we know that $[z,w]$ is contained in both $\Gamma_1$ and $\Gamma_2$. Then by Lemma \ref{lemma:PieceIntersection}, $[z,w]$ is a single piece, and thus $d(z,w)\leq M$. Therefore, $d(x_1,y_1)=d(x_1,z)+d(z,w)+d(w,y_1)\leq 5M+8\delta$, so $d(x_0,y_0)\leq 5M+10\delta$.

Since $x_0,y_0$ are taken arbitrarily, we obtained that the diameter of the intersection $N_\delta(\Gamma_1)\cap N_\delta(\Gamma_2)$ is at most $5M+10\delta$.
\end{proof}

By Proposition \ref{proposition:contraction} and Lemma \ref{lemma:ASatisfyAlpha1}, $\mathcal{A}$ is a strongly contracting system with \textit{bounded intersection property} considered in \cite{Y14}. The collection in Example \ref{Example:LeftCosetsInRelHyp} is such a system as well.

\begin{Lemma} \label{lemma:ASatisfyAlpha2}
$\mathcal{A}$ satisfies property $(\Lambda_2)$.
\end{Lemma}

\begin{proof}
Let $\alpha$ be a geodesic whose endpoints at distance at most $\frac{1}{3}|\alpha|$ from an embedded component $\Gamma_0\in\mathcal{A}$, where $|\alpha|$ is the length of $\alpha$.

If the intersection $\alpha\cap \Gamma_0$ is nonempty, we are done.
Now we assume that $\alpha\cap \Gamma_0=\emptyset$, then $\mathrm{diam}(\pi_{\Gamma_0}(\alpha))\leq 2M$ by Proposition \ref{proposition:contraction}. Let $a'\in \pi_{\Gamma_0}(\alpha_-),b'\in \pi_{\Gamma_0}(\alpha_+)$, then $d(a',b')\leq \mathrm{diam}(\pi_{\Gamma_0}(\alpha))\leq 2M$, and $d(\alpha_-,a')\leq \frac{1}{3}|\alpha|,d(\alpha_+,b')\leq\frac{1}{3}|\alpha|$. Therefore, $$|\alpha|=d(\alpha_-,\alpha_+)\leq d(\alpha_-,a')+d(a',b')+d(b',\alpha_+)\leq \frac{2}{3}|\alpha|+2M,$$
so $|\alpha|\leq 6M$, and thus $d(\alpha_-,a')\leq \frac{1}{3}|\alpha|\leq 2M$. So we always have that $\alpha\cap N_{2M+1}(\Gamma_0)\neq\emptyset$.

We have proved that for any set $\Gamma_0$ in $\mathcal{A}$, a geodesic with endpoints at distance at most one third of its length from $\Gamma_0$ intersects the $(2M+1)$-tubular neighborhood of $\Gamma_0$. So $\mathcal{A}$ satisfies property $(\Lambda_2)$.
\end{proof}

By \cite[Lemma 4.5]{DS05}, Properties $(\Lambda_1)$ and $(\Lambda_2)$ together imply property $(T_1)$.

%\begin{Lemma}  \label{Lemma:L1L2ImplyT1}
%Let $Y$ be a geodesic metric space and let %$\mathcal{B}=\{A_\lambda|\lambda\in\Lambda\}$ be a collection of subsets of $Y$ which satisfies properties $(\Lambda_1)$ and $(\Lambda_2)$. Then in every asymptotic cone $\mathrm{Con}_\omega(Y;(e_n),(l_n))$, the collection $\mathcal{B}_\omega$ satisfies property $(T_1)$.
%\end{Lemma}

\begin{Corollary} \label{corollary:ASatisfiesT1}
In every asymptotic cone $\mathrm{Con}_\omega(X;(e_n),(l_n))$ of $X=\mathscr{G}(G(\Gamma),S)$, the collection $\mathcal{A}$ satisfies property $(T_1)$.
\end{Corollary}

\subsection{$\mathcal{A}_\omega$ satisfies $(\Omega_3)$}

In this subsection, we shall omit mentioning the requirement $\omega$-almost surely in most statements. However, we need keep in mind that our discussion are always under such an assumption.

Since there are only finite forms of simple combinatorial triangles, we know from Theorem \ref{thm:ClassificationForTriangle} that:
\begin{Lemma} \label{lemma:OneShapeAlmostSurely}
Let $\Delta_n$ be a sequence of simple combinatorial triangles. Then $\Delta_n$ has one of the forms $\mathrm{I}_1$, $\mathrm{I}_2$, $\mathrm{I}_3$, $\mathrm{II}$, $\mathrm{III}$, $\mathrm{IV}$, $\mathrm{V}$ $\omega$-almost surely.
\end{Lemma}

Let $(Y,d_Y)$ be a geodesic metric space and let $\mathcal{B}$ be a collection of subsets of $Y$.
Recall that property $(\Omega_2)$ refers to:
\begin{itemize}
\item[$(\Omega_2)$] For any asymptotic cone $\mathrm{Con}_{\omega}(Y; (e_n),(l_n))$,
 any simple non-trivial bigon with edges limit geodesics is contained in a subset from $\mathcal{B}_\omega$.
\end{itemize}

\begin{Lemma}
$\mathcal{A}_{\omega}$ satisfies property $(\Omega_2)$.
\end{Lemma}

\begin{proof}
Given an asymptotic cone $\mathrm{Con}_{\omega}(X; (e_n), (l_n))$ of $X$, let $P$ be a non-trivial simple geodesic bigon in $\mathrm{Con}_{\omega}(X; (e_n), (l_n))$ whose edges $\alpha,\beta$ are limit geodesics.
Take $\theta>6M$, then by lemma \ref{lemma:FatQuadrangles}, $P$ is the limit of a sequence of $\omega$-almost surely $\theta$-fat simple geodesic quadrangles
\[Q_n=\alpha_n[\alpha_{n+},\beta_{n+}]\overline{\beta_n}[\beta_{n-},\alpha_{n-}]\]
such that $\alpha=\lim_\omega\alpha_n,\beta=\lim_\omega\beta_n$. Denote $\gamma_n=[\alpha_{n-},\beta_{n-}], \delta_n=[\alpha_{n+},\beta_{n+}]$.
We have that $d(\alpha_{n-},\beta_{n-})=|\gamma_n|$ and $d(\alpha_{n+},\beta_{n+})=|\delta_n|$ are of order $o(l_n)$ $\omega$-almost surely.

For each $n$, let $D_n$ be a diagram over $\langle S|R\rangle$ as in Lemma \ref{Lemma:CombPolygon} filling the quadrangle $Q_n$. % $\omega$-almost surely.

\vspace{1em}
\noindent\textbf{Case 1.} $(D_n)$ are non-degenerate, irreducible $\omega$-almost surely.

Since $D_n$ is simple, Lemma \ref{lemma:SpecialCombGeodQuadrangles} tells us that
there is a path $p_n\subseteq D_n$ either between $\alpha_n$ and $\beta_n$ or between $\gamma_n$ and $\delta_n$ which is a concatenation of at most $6$ interior arcs of $D_n$. By our choice of diagram $D_n$, the interior arcs of $D_n$ are pieces, thus the length of $p_n$ satisfies $|p_n| \leq6M$. If $p_n$ is between $\gamma_n$ and $\delta_n$ $\omega$-almost surely, %thus $p_n$ is also between them $\omega$-almost surely,
then $k_n=d(\gamma_n,\delta_n)\leq_\omega 6M$. Thus,
\[d(\alpha_{n-},\alpha_{n+})\leq_\omega d(\alpha_{n-},\beta_{n-})+k_n+d(\alpha_{n+},\beta_{n+})=o(l_n),\]
which contradicts with $d(\alpha_{n-},\alpha_{n+})=O(l_n)$.
However, if $p_n$ is between $\alpha_n$ and $\beta_n$ $\omega$-almost surely, it will contradict the requirements that $Q_n$ is $\theta$-fat and $\theta>6M$.
Therefore, this case can not happen.

\vspace{1em}
\noindent\textbf{Case 2.} $(D_n)$ are non-degenerate, reducible $\omega$-almost surely.

Since $D_n$ is simple, $D_n$ is not vertex reducible. Now assume that there is a face reduction. We suppose that $\Pi_n$ is a face with edges $e_n,e_n'\subset \Pi_n$ such that $D_n$ admits a face reduction at $(\Pi_n,e_n,e_n')$. Since $D_n$ is non-degenerate, there are four distinguished faces.

\noindent\textbf{Subcase 2.1.}
Suppose $e_n,e_n'$ are edges of two opposite sides $\omega$-almost surely. First, assume $e_n\subset\alpha_n$ and $e_n'\subset\beta_n$.
Choose $x_n\in e_n, x_n'\in e_n'$, and denote by $x=\lim_\omega x_n,x'=\lim_\omega x_n'$.
Face reduction at $(\Pi_n,e_n,e_n')$ gets two combinatorial triangles $\Delta_n,\Delta_n'$ with $\gamma_n\subset\Delta_n,\delta_n\subset\Delta_n'$. %$e=\lim_\omega e_n,e'=\lim_\omega e_n'$.
By Lemma \ref{lemma:OneShapeAlmostSurely}, $\Delta_n,\Delta_n'$ has one form $\omega$-almost surely.

If $\Delta_n$ is of form $\mathrm{II},\mathrm{III},\mathrm{IV},$ or $\mathrm{V}$, then there is a path from $\alpha_n$ to $\beta_n$ which is a concatenation of at most $4$ interior edges of $D_n$, contradicting the choice that $Q_n$ is $\theta$-fat.
By the same reason, $\Delta_n'$ cannot be one of form $\mathrm{II},\mathrm{III},\mathrm{IV},$ or $\mathrm{V}$ $\omega$-almost surely.

If $\Delta_n$ is of form $\mathrm{I}_1,\mathrm{I}_2$ or $\mathrm{I}_3$, then $[\alpha_-,x]_\alpha\cup[\beta_-,x']_\beta$ is contained in a set $A\in\mathcal{A}_\omega$.
Similarly, we can obtain that $[e,\alpha_+]_\alpha\cup[e',\beta_+]_\beta$ is contained in some set $A'\in\mathcal{A}_\omega$.
If $e=e'$, then either $e=\alpha_-$ or $e=\alpha_+$ since $P$ is simple, and thus $P=\alpha\cup\beta$ is contained in one subset from $\mathcal{A}_\omega$. Otherwise $e$ and $e'$ are two different intersection points of $A\cap A'$, then by Corollary \ref{corollary:ASatisfiesT1}, $\mathcal{A}_\omega$ satisfies property $(T_1)$, which imply $A=A'$.

Next, assume $e_n\subset\gamma_n$ and $e_n'\subset\delta_n$. Face reduction at $(\Pi_n,e_n,e_n')$ gets two combinatorial triangles $\Delta_n'',\Delta_n'''$ with $\alpha_n\subset\Delta_n'',\beta_n\subset\Delta_n'''$.

If $\Delta_n''$ is of form $\mathrm{II},\mathrm{III},\mathrm{IV},$ or $\mathrm{V}$, then there is a path from $\gamma_n$ to $\delta_n$ which is a concatenation of at most $4$ interior edges of $D_n$. Same arguments as Case 1 lead to a contradiction. Similarly, $\Delta_n'''$ cannot be one of form $\mathrm{II},\mathrm{III},\mathrm{IV},$ or $\mathrm{V}$ $\omega$-almost surely as well.

If $\Delta_n''$ is of form $\mathrm{I}_1,\mathrm{I}_2$ or $\mathrm{I}_3$, then $\alpha$ is contained in a set $A''\in\mathcal{A}_\omega$.
Similarly, we can obtain that $\beta$ is contained in some set $A'''\in\mathcal{A}_\omega$.
Since $\alpha_-=\beta_-$ and $\alpha_+=\beta_+$ are two different intersection points of $A''\cap A'''$, we have $A''=A'''$ by Corollary \ref{corollary:ASatisfiesT1} again.

\noindent\textbf{Subcase 2.2.}
Suppose $e_n,e_n'$ are edges of two adjacent sides. Without loss of generality, we assume $e_n\subseteq \alpha_n$ and $e_n'\subseteq \gamma_n$. Face reduction of $D_n$ at $(\Pi_n,e_n,e_n')$ gets a bigon and a quadrangle. Then there exists a single piece connecting $\alpha_n$ and $\gamma_n$.
However, for any such piece $p_n\subseteq D_n$ with $p_{n-}\in \alpha_n$ and $p_{n+}\in\gamma_n$, the $\omega$-limit of the triangle
\[[\alpha_{n-},p_{n-}]_{\alpha_n}\cup p_n\cup[p_{n+},\alpha_{n-}]_{\overline{\gamma_n}}\]
is a single point $\alpha_-$.
Let $a_n$ be the innermost piece between $\alpha_n$ and $\gamma_n$ (corresponding to $\alpha_{n-}$),
i.e., there is no single piece between $[a_{n-},\alpha_{n+}]_{\alpha_n}$ and $[a_{n+},\gamma_{n+}]_{\gamma_n}$. Symmetrically, let $a_n',b_n,b_n'$ be the innermost piece corresponding to $\alpha_{n+},\beta_{n-},\beta_{n+}$, respectively, if it exists. Let $D_n''$ be the subdiagram of $D_n$ with sides
\[[a_{n-},a_{n-}']_{\alpha_n},a_n'[a_{n+}',b_{n-}']_{\delta_n}b_n',[b_{n+}',b_{n+}]_{\overline{\beta_n}},
\overline{b_n'}[b_{n-},a_{n+}]_{\overline{\gamma_n}}\overline{a_n}.\]
\begin{figure}[H]
  \centering
  \includegraphics[width=0.4\textwidth]{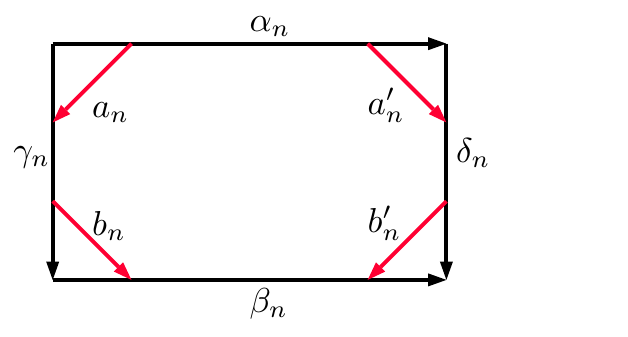}
  \caption{Subdiagram}\label{figure:subdiagram}
  \vspace{-0.5cm}
\end{figure}
The new diagram $D_n''$ is still a combinatorial quadrangle, and $D_n''$ has no face reduction of the present subcase.
We can replace $D_n$ by $D_n''$, and proceed the discussion either in Subcase 2.1 or in Case 1. Therefore the result also holds for this subcase.

\vspace{1em}
\noindent\textbf{Case 3.} $(D_n)$ are degenerate $\omega$-almost surely. Then $(D_n)$ are simple combinatorial triangles $\omega$-almost surely.

If $D_n$ is of form $\mathrm{II},\mathrm{III},\mathrm{IV},$ or $\mathrm{V}$, then there is a path from $\gamma_n$ to $\delta_n$ which is a concatenation of at most $4$ interior edges of $D_n$. The same argument as in case 1 then tells us this cannot happen.

If $D_n$ is of form $\mathrm{I}_1,\mathrm{I}_2$ or $\mathrm{I}_3$, then the same argument as in Case 2 implies that $P=\alpha\cup\beta$ is contained in one set from $\mathcal{A}_\omega$.
\end{proof}

\begin{Lemma}\label{lemma:EndpointPiece}
Let $Y$ be a geodesic metric space and let $\mathcal{B}$ be a collection of convex subsets of $Y$ which satisfies property $(\Omega_2)$. Then for any asymptotic cone $\mathrm{Con}_\omega(Y; (e_n), (l_n))$ of $Y$, any limit geodesic $\alpha$ with endpoints contained in a set $B\in\mathcal{B}_\omega$ is contained in $B$.
\end{Lemma}

\begin{proof}
Set $B=\lim_\omega B_n$. Then there exist $a_n,b_n\in B_n$, so that $\alpha_-=\lim_{\omega}a_n,\alpha_+=\lim_{\omega}b_n$. Since $B_n$ is convex, the geodesic $[a_n,b_n]\subseteq B_n$, and thus the limit geodesic $\beta=\lim_{\omega}[a_n,b_n]\subseteq B$. Moreover, the edges of the simple bigon $P$ formed by $\alpha$ and $\beta$ are limit geodesics, therefore, such bigon $P$ is contained in a subset from $\mathcal{B}_\omega$ since $\mathcal{B}$ satisfies property $(\Omega_2)$.
Then by Lemma \ref{lemma:OneImplyTheOther}, we have $\alpha\subseteq B$.
\end{proof}

\begin{Lemma}\label{lemma:SimpleTriangle}
Assume that the simple triangle $\Delta$ is contained in the $\omega$-limit of a sequence simple triangles $T_n=\alpha_n\beta_n\gamma_n$, i.e., $\Delta\subseteq\lim_\omega T_n$. Assume further that the intersection of the three edges $\lim_\omega\alpha_n$, $\lim_\omega\beta_n$, $\lim_\omega\gamma_n$ of $\lim_\omega T_n$ is empty. Then $\Delta$ is contained in a set $A\in \mathcal{A}_\omega$.
\end{Lemma}

\begin{proof}
Let $(D_n)$ be diagrams over $\langle S|R\rangle$ as in Lemma \ref{Lemma:CombPolygon} filling the triangles $(T_n)$ $\omega$-almost surely.
By Lemma \ref{lemma:OneShapeAlmostSurely}, $(D_n)$ have one form $\omega$-almost surely.

It is obvious that the $\omega$-limit of a sequence single pieces is either a single point or empty. If $(D_n)$ are of form $\mathrm{III},\mathrm{IV}$, or $\mathrm{V}$ $\omega$-almost surely, then the three edges of $\lim_\omega T_n$ always have a common point, contradicting the hypothesis.
Therefore, $(D_n)$ are of form $\mathrm{I}_1,\mathrm{I}_2,\mathrm{I}_3$ or $\mathrm{II}$ $\omega$-almost surely. Since $\Delta$ is simple, it must be contained in the $\omega$-limit of a sequence of boundaries of single faces $(\Pi_n)$ of $(D_n)$. Each $\partial\Pi_n$ is contained in an embedded component $\Gamma_n$, thus $\Delta$ is contained in the set $\lim_\omega\Gamma_n\in \mathcal{A}_\omega$.
\end{proof}

\begin{Lemma} \label{lemma:ASatisfyPi3}
$\mathcal{A}$ satisfies property $(\Omega_3)$: for any asymptotic cone $\mathrm{Con}_{\omega}(Y; (e_n),(l_n))$,
 any simple non-trivial triangles with edges limit geodesics is contained in a subset from $\mathcal{B}_\omega$.
\end{Lemma}

\begin{proof}
Given an asymptotic cone $\mathrm{Con}_{\omega}(X; (e_n), (l_n))$ of $X$, let $\Delta$ be a non-trivial simple geodesic triangle in $\mathrm{Con}_{\omega}(X; (e_n), (l_n))$ whose edges $\alpha_1,\beta_1,\gamma_1$ are limit geodesics. Therefore, $\alpha_1,\beta_1,\gamma_1$ can be presented as
%By lemma \ref{lemma:FatQuadrangles}, $\Delta$ is the limit of a sequence $H_n=a_nb_n'b_nc_n'c_na_n'$ of $(\theta,\nu)$-fat $\omega$-almost surely simple geodesic hexagons and
\[\gamma_1=\lim_\omega[a_n,b_n'],\alpha_1=\lim_\omega[b_n,c_n'],\beta_1=\lim_\omega[c_n,a_n'],\]
and $d(a_n,a_n'),d(b_n,b_n')$ and $d(c_n,c_n')$ are of order $o(l_n)$ $\omega$-almost surely.

We consider the $\omega$-limit of triangles $T_n=[a_n,b_n][b_n,c_n][c_n,a_n]$. Set $a=\lim_\omega a_n$, $b=\lim_\omega b_n$, $c=\lim_\omega c_n$.
And set
\[\gamma_2=\lim_\omega[a_n,b_n], \beta_2=\lim_\omega[c_n,a_n], \alpha_2=\lim_\omega[b_n,c_n].\]

Let $x$ be the furthest point in the intersection of $\gamma_2$ and $\overline{\beta_2}$ along $\gamma_2$, i.e., $x\in \gamma_2\cap\overline{\beta_2}$ and $(x,b]_{\gamma_2}\cap\overline{\beta_2}=\emptyset$.
Similarly, let $y$ (resp. $z$) be the furthest point in the intersection of $\alpha_2$ and $\overline{\gamma_2}$ (resp. $\beta_2$ and $\overline{\alpha_2}$) along $\alpha_2$ (resp. $\beta_2$).

If $x\neq a$, then $[a,x]_{\gamma_2}\cup[a,x]_{\overline{\beta_2}}$ is contained in a subset $A_1\in\mathcal{A}_\omega$ by Lemma \ref{lemma:FourGeodesics}. Symmetrically, if $y\neq b$ (resp. $z\neq c$), then $[b,y]_{\alpha_2}\cap[b,y]_{\overline{\gamma_2}}$ (resp. $[c,z]_{\beta_2}\cup[c,z]_{\overline{\alpha_2}}$) is contained in a subset $A_2\in\mathcal{A}_\omega$ (resp. $A_3\in\mathcal{A}_\omega$).

Consider the following $6$ numbers:
\begin{align*}
& d_\omega([a,x]_{\gamma_2},\alpha_2), d_\omega([a,x]_{\overline{\beta_2}},\alpha_2), d_\omega([b,y]_{\alpha_2},\beta_2), \\
& d_\omega([b,y]_{\overline{\gamma_2}},\beta_2), d_\omega([c,z]_{\beta_2},\gamma_2), d_\omega([c,z]_{\overline{\alpha_2}},\gamma_2).
\end{align*}

\vspace{1em}
\noindent\textbf{Case 1.}\,\,
None of the $6$ numbers equals to $0$.
Then the positions of $x,y,z$ must satisfy that $y\in (x,b]_{\gamma_2}$, $z\in (x,c]_{\overline{\beta_2}}$ and $z\in (y,c]_{\alpha_2}$ (See the left one of Figure \ref{figure:case1}). Moreover, the triangle $\Delta'=[x,y]_{\gamma_2}[y,z]_{\alpha_2}[z,x]_{\beta_2}$ is simple and the intersection $[x,y]_{\gamma_2}\cap[y,z]_{\alpha_2}\cap[z,x]_{\beta_2}$ is empty.

We can choose $[x_n,y_n']\subseteq [a_n,b_n]$, $[y_n,z_n']\subseteq[b_n,c_n]$, $[z_n,x_n']\subseteq[c_n,a_n]$ $\omega$-almost surely such that $[x_n,y_n'], [y_n,z_n'], [z_n,x_n']$ are piecewise disjoint and
\[\lim_\omega[x_n,y_n']=[x,y]_{\gamma_2}, \lim_\omega[y_n,z_n']=[y,z]_{\alpha_2}, \lim_\omega[z_n,x_n']=[z,x]_{\beta_2}.\]

Let $p_n$ (resp. $q_n,r_n$) be the furthest point in the intersection of $[a_n,b_n]$ and $\overline{[c_n,a_n]}$ (resp, $[b_n,c_n]$ and $\overline{[a_n,b_n]}$, $[c_n,a_n]$ and $\overline{[b_n,c_n]}$) along $[a_n,b_n]$ (resp. $[b_n,c_n]$, $[c_n,a_n]$).
Since $d_\omega([a,x]_{\gamma_2},\alpha_2)>0$, we have that $[p_n,x_n]\subseteq[a_n,b_n]$ is disjoint from $[b_n,c_n]$. By the same reason, we know that the triangle $[p_n,q_n][q_n,r_n][r_n,p_n]\subseteq T_n$ is simple (See the right one of Figure \ref{figure:case1}).

Therefore, we have proved that $\Delta'$ is contained in the $\omega$-limit of a sequence simple triangles $[p_n,q_n][q_n,r_n][r_n,p_n]$. Then by Lemma \ref{lemma:SimpleTriangle}, $\Delta'$ is contained in a subset $A_0\in\mathcal{A}_\omega$.

If $x\neq a$, we have $x\notin\gamma_1$ or $x\notin\beta_1$ since $\Delta=\alpha_1\beta_1\gamma_1$ is simple.
Without loss of generality, we may assume that $x\notin\gamma_1$. By \cite[Lemma 3.7]{D09}, $x$ is in the interior of some simple bigon formed by $\gamma_1$ and $\gamma_2$, then property $(\Omega_2)$ tells us that this bigon is in contained a subset $A_1'\in\mathcal{A}_\omega$. Since $A_1'\cap A_0$ and $A_1'\cap A_1$ contain non-trivial sub-arcs of $\gamma_2$, property $(T_1)$ (Corollary \ref{corollary:ASatisfiesT1}) implies that $A_0=A_1'=A_1$.

Symmetrically, if $y\neq b$ (resp. $z\neq c$),  we have that $A_0=A_2$ (resp. $A_0=A_3$). Therefore, we have that $A_0=A_1=A_2=A_3$. So $\Delta'$ is contained in $A_0$, and thus $\Delta$ is contained in $A_0$ as well by Lemma \ref{lemma:OneImplyTheOther}.

If $x=a$, $y=b$ or $z=c$, we can just replace the corresponding $A_i$ by $A_0$. The result still holds.

\begin{figure}[H]
\captionsetup[subfigure]{labelformat=empty}
  \centering\hfill%

\begin{subfigure}[h]{.45\textwidth}\centering
\includegraphics[width=0.9\textwidth]{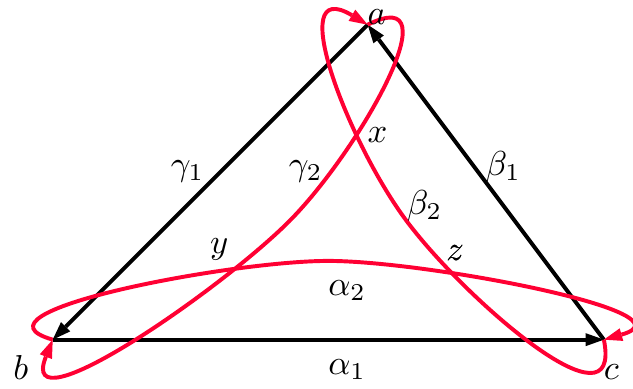}
% \caption{Position} \label{figure:position}
   \end{subfigure}\hfill%
  \begin{subfigure}[h]{.55\textwidth}\centering
\includegraphics[width=0.95\textwidth]{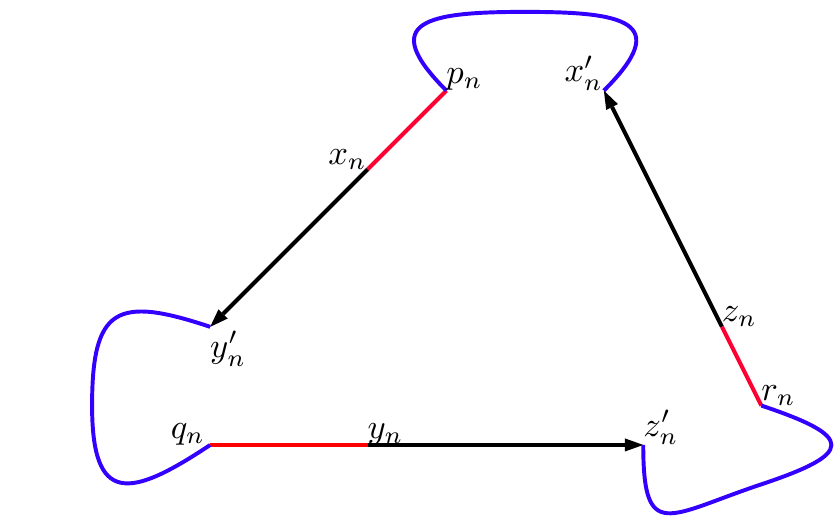}
% \caption{Simple triangle} \label{figure:simple}
  \end{subfigure}\hfill

\caption{Case 1.}
\label{figure:case1}
\end{figure}

\vspace{1em}
\noindent\textbf{Case 2.}\,\,One of the $6$ numbers equals to $0$.

Without loss of generality, we set $d_\omega([a,x]_{\gamma_2},\alpha_2)=0$.
Now let $y''$ be the furthest point in the intersection of $\overline{\gamma_2}$ and $\alpha_2$ along $\overline{\gamma_2}$, then $y''\in[a,x]_{\gamma_2}$. If $y''=b$, then $\gamma_2\subseteq A_1$, otherwise $[b,y'']_{\overline{\gamma_2}}\cup[b,y'']_{\alpha_2}$ is contained in a subset $A_2'\in\mathcal{A}_\omega$ by Lemma \ref{lemma:FourGeodesics}.
Therefore $\gamma_2\subseteq A_1\cup A_2'$ and $x\in A_1\cap A_2'$. Since $\Delta$ is simple, we have that $x\notin\gamma_1$ or $x\notin\beta_1$. The same argument as in the previous case then tells us that $A_1=A_2'$. Thus we always have that $\gamma_2\subseteq A_1$.  By Lemma \ref{lemma:EndpointPiece}, we know that $\gamma_1\subseteq A_1$ as well.

If $c\in A_1$, then by Lemma \ref{lemma:EndpointPiece} again, we know $\alpha_1,\beta_1\subseteq A_1$. Thus $\Delta$ is contained in $A_1$.

Next we assume that $c\notin A_1$. Let $[c,\widetilde{a})_{\beta_1}$ be the maximal segment of $\beta_1$ which is in outside of $A_1$, i.e. $\widetilde{a}\in A_1$ and $[c,\widetilde{a})_{\beta_1}\cap A_1=\emptyset$ (Notice $A_1$ is closed, such $\widetilde{a}$ exists). Similarly, let $[c,\widetilde{b})_{\overline{\alpha_1}}$ be the maximal segment of $\overline{\alpha_1}$ which is in outside of $A_1$.

By our definition of $\mathcal{A}_\omega$, there exists a sequence embedded components $(\Gamma_n)$ such that $A_1=\lim_{\omega}\Gamma_n$.

We can choose $\widetilde{a}_n\in[c_n,a_n']$ and $\widetilde{b}_n\in[b_n,c_n']$ $\omega$-almost surely, such that $\lim_\omega \widetilde{a}_n=\widetilde{a}$, $\lim_\omega\widetilde{b}_n=\widetilde{b}$, $[c_n,\widetilde{a}_n]\subseteq [c_n,a_n']$ and $[\widetilde{b}_n,c_n']\subseteq [b_n,c_n']$ are disjoint from $N_M(\Gamma_n)$.
Since $c\notin A_1$ and $\lim_\omega c_n=c=\lim_\omega c_n'$, we have that $[c_n',c_n]$ is disjoint from $N_M(\Gamma_n)$.
Since $\lim_\omega \widetilde{a}_n=\widetilde{a}$ and $\widetilde{a}\in A_1$, we obtain that $d(\widetilde{a}_n,\Gamma_n)=o(l_n)$.
We also have that $d(\widetilde{b}_n,\Gamma_n)=o(l_n)$ via the same reason.
By Proposition \ref{proposition:contraction}, $\Gamma_n$ is $(2M,0)$-contracting, so
\[\mathrm{diam}(\pi_{\Gamma_n}([\widetilde{b}_n,c_n']))\leq M, \mathrm{diam}(\pi_{\Gamma_n}([c_n',c_n]))\leq M, \mathrm{diam}(\pi_{\Gamma_n}([c_n,\widetilde{a}_n]))\leq M.\]
Thus, we have that
\begin{align*}
d(\widetilde{b}_n,\widetilde{a}_n)\leq& d(\widetilde{b}_n,\Gamma_n)+ \mathrm{diam}(\pi_{\Gamma_n}([\widetilde{b}_n,c_n']))+\mathrm{diam}(\pi_{\Gamma_n}([c_n',c_n])) \\
& + \mathrm{diam}(\pi_{\Gamma_n}([c_n,\widetilde{a}_n]))+d(\widetilde{a}_n,\Gamma_n)  \\
\leq& o(l_n)+6M+o(l_n)=o(l_n).
\end{align*}

But $\widetilde{a}\neq\widetilde{b}$ imply that $d(\widetilde{b}_n,\widetilde{a}_n)=O(l_n)$. This is a a contradiction. So our assumption $c\notin A_1$ is not true.

Hence, $\Delta$ is always contained in a subset from $\mathcal{A}_\omega$. We have completed the proof.
\end{proof}

\begin{proof}[Proof of Theorem \ref{thm:RelativeHyperbolicity}]
Assume that  all pieces of $\Gamma$ have uniformly bounded length. According to Lemma \ref{lemma:tree-graded}, the properties ($\Lambda_1$),  ($\Lambda_2$) and ($\Omega_3$) are verified by Lemmas \ref{lemma:ASatisfyAlpha1}, \ref{lemma:ASatisfyAlpha2} and \ref{lemma:ASatisfyPi3} respectively, so  $\mathscr{G}(G(\Gamma),S)$ is asymptotically tree-graded with respect to the collection $\mathcal{A}$ of all embedded components of $\Gamma$. For the converse direction, if $\mathscr{G}(G(\Gamma),S)$ is asymptotically tree-graded with respect to $\mathcal{A}$, then the ``only if'' part of Lemma \ref{lemma:tree-graded} tells us that $\mathcal{A}$ satisfies property ($\Lambda_1$), in particular, the intersection of two different components are uniformly bounded.  Lemma \ref{lemma:PieceIntersection} thus implies  that the pieces of $\Gamma$ are uniformly bounded.
\end{proof}

\begin{Examples}
\vspace{.5em}
\noindent\textbf{1.}
Let $G_1, G_2$ be two groups generated by $S_1, S_2$ respectively. Let $\Gamma_i$ be the Cayley graph of $(G_i,S_i)$ for $i=1,2$ and let $\Gamma=\Gamma_1\sqcup \Gamma_2$. Then $\Gamma$ has no pieces with length bigger than $0$. Obviously,  $G(\Gamma)=G_1\ast G_2$ is  hyperbolic relative to $G_1,G_2$.

\vspace{.5em}
\noindent\textbf{2.}
Let $S=\{a,b,x,y,z,w\}$. And let
\[G_1=\langle x,y|[x,y]\rangle, G_2=\langle z,w|[z,w]\rangle, X_1=\mathscr{G}(G_1,\{x,y\}), X_2=\mathscr{G}(G_2,\{z,w\}).\]

Let $\Gamma_1$ be the graph obtained from the Cayley graph $X_1$ as follows: for each $g\in G_1$, attach  a  simple path $p$ labelled by $a^2b^3$ to the pair of vertices $(g,gx^{15}y^{15})$ in $X_1$.

Similarly, let $\Gamma_2$ be the graph constructed from $X_2$ so that each pair of vertices $(g, gz^{15}w^{15})$ is connected by a simple path labeled by $a^2b^2$.

Denote $\Gamma=\Gamma_1\sqcup \Gamma_2$. Then the pieces of $\Gamma$ are those subpaths of the simple path labelled either by $a^2b^2a^2$
or by its inverse $a^{-2}b^{-2}a^{-2}$. Thus, $\Gamma$ is $Gr'(\frac{1}{6})$-labelled. Then for the basepoints $1_{G_1}\in \Gamma_1$, $1_{G_2}\in \Gamma_2$ and $1_{G}\in \mathscr{G}(G(\Gamma),S)$,  the image of the label-preserving map $(\Gamma_i,1_{G_i})\rightarrow (\mathscr{G}(G(\Gamma),S),1_G)$ gives embedded components $\Gamma_i'$ in $\mathscr{G}(G(\Gamma),S)$ by Lemma \ref{lemma:EmbedConvex}. Furthermore the stabilizer of $\Gamma_1'$ is the subgroup of $G(\Gamma)$ generated by $\{x,y,a^2b^3\}$, which is $G_1$ since $a^2b^3=x^{15}y^{15}$. Similarly, the stabilizer of $\Gamma_2'$ is $G_2$.
Thus by Theorem \ref{thm:RelativeHyperbolicity}, $G(\Gamma)$ is hyperbolic relative to $G_1,G_2$.
\begin{figure}[H]
  \centering
  \includegraphics[width=0.7\textwidth]{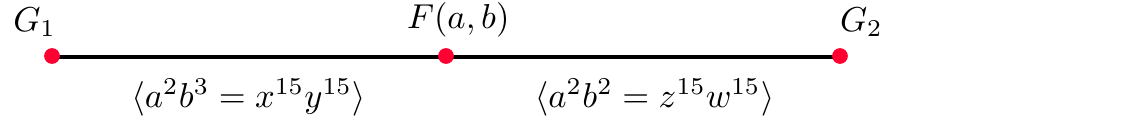}
  \caption{The graph of groups}\label{figure:graphofgroup}
  \vspace{-0.5cm}
\end{figure}
In fact, $G(\Gamma)$ can be presented as a fundamental group of a graph of group with vertex groups $\{G_1,G_2,F(a,b)\}$ and edge groups
$\{\langle a^2b^3=x^{15}y^{15}\rangle\cong\mathbb{Z},\langle a^2b^2=z^{15}w^{15}\rangle\cong\mathbb{Z}\}$, where $F(a,b)$ is the free group on $\{a,b\}$. (See Figure \ref{figure:graphofgroup})
\end{Examples}

\end{document}